\documentclass[oneside,english]{amsart}
\usepackage[T1]{fontenc}
\usepackage[latin9]{inputenc}
\usepackage{verbatim}
\usepackage{amstext}
\usepackage{amsthm}
\usepackage{amssymb}
\usepackage{esint}

\makeatletter
\numberwithin{equation}{section}
\numberwithin{figure}{section}
\theoremstyle{plain}
\newtheorem{thm}{\protect\theoremname}
  \theoremstyle{plain}
  \newtheorem{lem}{\protect\lemmaname}
  \theoremstyle{remark}
  \newtheorem{claim}{\protect\claimname}

\makeatother

\usepackage{babel}
  \providecommand{\claimname}{Claim}
  \providecommand{\lemmaname}{Lemma}
\providecommand{\theoremname}{Theorem}

\begin{document}

\title{Interior Hessian estimates for sigma-2 equations in dimension three}

\author{GUOHUAN QIU}
\begin{abstract}
We prove a priori interior $C^{2}$ estimate for $\sigma_{2}=f$ in
$\mathbb{R}^{3}$, which generalizes Warren-Yuan's result \cite{warren2009hessian}. 
\end{abstract}


\date{\today}

\address{Department of Mathematics and Statistics, McGill University, 805
Sherbrooke O, Montreal, Quebec, Canada, H3A 0B9}

\email{guohuan.qiu@mail.mcgill.ca}

\thanks{Research of the author was supported by CRC postdoctoral fellowship
at McGill University.}

\maketitle

\section{Introduction}

The interior regularity for solutions of the following $\sigma_{2}$-
Hessian equations is a longstanding problem in fully nonlinear partial
differential equations, 
\begin{equation}
\sigma_{2}(D^{2}u)=f(x,u,Du)>0,\begin{array}{ccc}
 & x\in B_{1}\subset\mathbb{R}^{n}\end{array}\label{eq:sigma2}
\end{equation}

where $\sigma_{k}$ the $k-$th elementary symmetric function for
$1\le k\leq n$. 

Heinz \cite{heinz1959elliptic} first derived this interior estimate
in two dimension. For higher dimensional Monge-Ampere equations, Pogorelov
constructed his famous counter-examples even for $f$ constant and
convex solutions in \cite{pogorelov1978minkowski}. Caffarealli-Nirenberg-Spruck
studied more general fully nonlinear equations such as $\sigma_{k}$
equations in their seminal work \cite{caffarelli1985dirichlet}. And
Urbas also constructed counter-examples in \cite{urbas1990existence}
with $k\geq3$. Because of the counter-examples in $k\geq3$, the
best we can expect is the Pogorelov type interior $C^{2}$ estimates
which were derived in \cite{pogorelov1978minkowski,chou2001variational}
and see \cite{guan2015global,li2016interior} for more general form.
So people in this field want to know whether the interior $C^{2}$
estimate for $\sigma_{2}$ equations holds or not for $n\geq3$. Moreover,
this equation serves as a simply model for scalar curvature equations
in hypersurface geometry 
\begin{equation}
\sigma_{2}(\kappa_{1}(x),\cdots,\kappa_{n}(x))=f(x,\nu(x))>0,\begin{array}{ccc}
 & x\in B_{1}\subset\mathbb{R}^{n}\end{array}\label{eq:scalar}
\end{equation}

here $\kappa_{1},\cdots,\kappa_{n}$ are the principal curvatures
and $\nu$ the normal of the given hypersurface as a graph over a
ball $B_{1}\subset\mathbb{R}^{n}$. A major breakthrough was made
by Warren-Yuan \cite{warren2009hessian}, they obtained $C^{2}$ interior
estimate for the equation 
\begin{equation}
\sigma_{2}(D^{2}u)=1.\begin{array}{ccc}
 & x\in B_{1}\subset\mathbb{R}^{3}\end{array}\label{eq:specialLag}
\end{equation}
The purely interior $C^{2}$ estimates for solutions of equations
(\ref{eq:sigma2}) and (\ref{eq:scalar}) with certain convexity constraints
were obtained recently by McGonagle-Song-Yuan in \cite{MSY} and Guan
and the author in \cite{GuanQiu}. 

Now we state our main result in this paper 
\begin{thm}
Let $u$ be a smooth solution to (\ref{eq:sigma2}) on $B_{10}\subset\mathbb{R}^{3}$
with $\Delta u>0$. Then we have 
\begin{equation}
\sup_{B_{\frac{1}{2}}}|D^{2}u|\leq C,
\end{equation}
where $C$ depends only on $\|f\|_{C^{2}}$, $\|\frac{1}{f}\|_{L^{\infty}}$
and $||u||_{C^{1}}$.
\end{thm}
In order to introduce our idea, let us briefly review the ideas for
attacking this problem so far. In two dimensional case, Heinz used
the Uniformization theorem to transform this interior estimate for
Monge-Ampere equation into the regularity of an elliptic system and
univalent of this mapping, see also \cite{heinz1956certain,lu2016weyl}
for more details. An anothor interesting proof using only maximum
principle was given by Chen-Han-Ou in \cite{chen2015interior,chen2016interior}.
Our quantity (\ref{eq:test}) as in \cite{GuanQiu} can give a new
proof of Heinz. The restriction for these methods is that they need
some convexity conditions which are not available in higher dimension.
In $\mathbb{R}^{3}$, a key observation made in \cite{warren2009hessian}
is that equation (\ref{eq:specialLag}) is exactly the special Lagrangian
equation which stems from the special Lagrangian geometry \cite{harvey1982calibrated}.
And an important property for the special Lagrangian equation is that
the Lagrangian graph $(x,Du)\subset\mathbb{R}^{n}\times\mathbb{R}^{n}$
is a minimal surface which has mean value inequality. So Wang-Yuan
\cite{wang2014hessian} can also prove interior Hessian estimates
for higher dimensional special Lagrangian equations. The new observation
in this paper is that the graph $(x,Du)$ where $u$ satisfied equation
(\ref{eq:sigma2}) can be viewed as a submanifold in $(\mathbb{R}^{3}\times\mathbb{R}^{3},fdx^{2}+dy^{2})$
with bounded mean curvature. Then instead of using Michael-Simon's
mean value inequality \cite{michael1973sobolev} as in \cite{warren2009hessian},
we prove a mean value inequality for Riemannian submanifolds which
is implied in Hoffman-Spruck's paper \cite{hoffman1974sobolev} in
order to remove the convexity condition in \cite{GuanQiu}. There
also have sobolev inequalities as in \cite{michael1973sobolev} and
\cite{hoffman1974sobolev}. But it seems not easy to estimate the
additional term in the sobolev inequalities by using only Warren-Yuan's
arguement. The other innovation part is that we can avoid using the
sobolev inequality by combining our new maximum principle method in
\cite{GuanQiu} to solve the problem for the equations (\ref{eq:sigma2})
in $\mathbb{R}^{3}$. 

The scalar curvature equations and the higher dimensional case for
$\sigma_{2}$ Hessian equations are still open to us.

\section{Preliminary Lemmas}

We call a solution of the equation (\ref{eq:sigma2}) is admissible,
if $u$ is smooth and $D^{2}u\in\Gamma_{2}:=\{A:\sigma_{1}(A)>0,\sigma_{2}(A)>0\}$.
It follows from \cite{caffarelli1985dirichlet}, if $D^{2}u\in\Gamma_{2}$,
then $\sigma_{2}^{ij}:=\frac{\partial\sigma_{2}}{\partial u_{ij}}$
is positive definite. So the Hessian estimates can be reduced to the
estimate of $\Delta u$ due to the following fact 
\begin{equation}
\max|D^{2}u|\leq\Delta u.\label{eq:lap}
\end{equation}
In the rest of this article, we will denote $C$ to be constant under
control (depending only on $\|f\|_{C^{2}}$, $\|\frac{1}{f}\|_{L^{\infty}}$
and $\|u\|_{C^{1}}$), which may change line by line.  
\begin{lem}
\label{lem1}Suppose $u$ satisfies the equation (\ref{eq:sigma2}),
we have the following equations 
\begin{equation}
\sigma_{2}^{kl}u_{kli}=f_{i},\label{eq:f1}
\end{equation}
and 
\begin{equation}
\sigma_{2}^{kl}u_{klii}+\sum_{k\neq l}u_{kki}u_{lli}-\sum_{k\neq l}u_{kli}^{2}=f_{ii}.\label{eq:f2}
\end{equation}
If $f$ is a form with gradient term, there are estimates

\begin{equation}
|Df|\leq C(1+\Delta u),\label{eq:gradient}
\end{equation}
and 
\begin{equation}
-C(1+\Delta u)^{2}+\sum_{p}f_{u_{p}}u_{pij}\leq f_{ij}\leq C(1+\Delta u)^{2}+\sum_{p}f_{u_{p}}u_{pij}\label{eq:Hessianf}
\end{equation}
\end{lem}
\begin{proof}
The equations (\ref{eq:f1}) and (\ref{eq:f2}) follows from twice
differential of the equation $\sigma_{2}(D^{2}u)=f$ and the elementary
fact that 

\begin{equation}
\begin{array}{cc}
\frac{\partial^{2}\sigma_{2}}{\partial u_{kk}\partial u_{ll}}=1 & for\,\,k\neq l,\\
\frac{\partial^{2}\sigma_{2}}{\partial u_{kl}\partial u_{kl}}=-1 & for\,\,k\neq l,\\
\frac{\partial^{2}\sigma_{2}}{\partial u_{pq}\partial u_{kl}}=0 & otherwise.
\end{array}
\end{equation}
Moreover, by (\ref{eq:lap}) and these direct computations 
\begin{equation}
f_{i}=f_{x_{i}}+f_{u}u_{i}+f_{u_{p}}u_{pi},
\end{equation}

and 
\begin{eqnarray}
f_{ij} & = & f_{x_{i}x_{j}}+f_{x_{i}u}u_{j}+f_{x_{i}u_{p}}u_{pj}+f_{ux_{j}}u_{i}+f_{uu}u_{i}u_{j}+f_{uu_{p}}u_{pj}u_{i}\\
 &  & +f_{u}u_{ij}+f_{u_{p}x_{j}}u_{pi}+f_{u_{p}u}u_{j}u_{pi}+f_{u_{p}u_{l}}u_{lj}u_{pi}+f_{u_{p}}u_{pij},
\end{eqnarray}

we get the estimates (\ref{eq:gradient}) and (\ref{eq:Hessianf}). 
\end{proof}
The second lemma is from \cite{lin1994some}. 
\begin{lem}
\label{lemma2} Suppose $W\in\Gamma_{2}$ is diagonal and $W_{11}\ge\cdots\ge W_{nn}$,
then there exist $c_{1}>0$ and $c_{2}>0$ depending only on $n$
such that 
\begin{equation}
\sigma_{2}^{11}(W)W_{11}\geq c_{1}\sigma_{2}(W),\label{tildec0}
\end{equation}
and for any $j\geq2$ 
\begin{equation}
\sigma_{2}^{jj}(W)\geq c_{2}\sigma_{1}(W).\label{s222}
\end{equation}

\end{lem}
Let us consider the quantity of $b(x):=\log\Delta u$, we have
\begin{lem}
\label{lem3} If $u$ are admissible solutions of the equations (\ref{eq:sigma2})
in $\mathbb{R}^{3}$, we have
\begin{equation}
\sigma_{2}^{ij}b{}_{ij}\geq\frac{1}{100}\sigma_{2}^{ij}b{}_{i}b{}_{j}-C\Delta u+\sum_{i}f_{u_{i}}b_{i}.\label{eq:logb}
\end{equation}

\end{lem}

\begin{proof}
We may assume that $\{u_{ij}\}$ is diagonal. The differential equation
of $b$ by using (\ref{eq:f2}) is 
\begin{eqnarray}
A:=\sigma_{2}^{ij}b_{ij}-\epsilon\sigma_{2}^{ij}b_{i}b_{j} & = & \frac{\sigma_{2}^{ij}(\Delta u){}_{ij}}{\sigma_{1}}-\frac{(1+\epsilon)\sigma_{2}^{ii}(\sum_{k}u_{kki})^{2}}{\sigma_{1}^{2}}\\
 & = & \frac{\sum\limits _{i}(\sum_{k\neq p}u_{kpi}^{2}-\sum_{k\neq p}u_{kki}u_{ppi})+\Delta f}{\sigma_{1}}\\
 &  & -\frac{(1+\epsilon)\sigma_{2}^{ii}(\sum_{k}u_{kki})^{2}}{\sigma_{1}^{2}}.
\end{eqnarray}

We use (\ref{eq:f1}) to substitute terms with $u_{iii}$ in $A$,
\begin{eqnarray}
A & = & \frac{6u_{123}^{2}}{\sigma_{1}}+\frac{2\sum_{k\neq p}u_{kpp}^{2}}{\sigma_{1}}-\frac{2\sum_{k\neq p}u_{kkp}u_{ppp}}{\sigma_{1}}\\
 &  & -\frac{2u_{113}u_{223}+2u_{112}u_{332}+2u_{221}u_{331}}{\sigma_{1}}\\
 &  & +\frac{\triangle f}{\sigma_{1}}-\frac{(1+\epsilon)\sigma_{2}^{ii}(\sum_{k\neq i}u_{kki}+u_{iii})^{2}}{\sigma_{1}^{2}}\\
 & \geq & \frac{2(u_{211}^{2}+u_{311}^{2}+u_{122}^{2}+u_{322}^{2}+u_{133}^{2}+u_{233}^{2})}{\sigma_{1}}\\
 &  & -\frac{2(u_{221}+u_{331})(f_{1}-\sigma_{2}^{22}u_{221}-\sigma_{2}^{33}u_{331})}{\sigma_{1}\sigma_{2}^{11}}\\
 &  & -\frac{2(u_{112}+u_{332})(f_{2}-\sigma_{2}^{11}u_{112}-\sigma_{2}^{33}u_{332})}{\sigma_{1}\sigma_{2}^{22}}\\
 &  & -\frac{2(u_{113}+u_{223})(f_{3}-\sigma_{2}^{11}u_{113}-\sigma_{2}^{22}u_{223})}{\sigma_{1}\sigma_{2}^{33}}\\
 &  & -\frac{2u_{113}u_{223}+2u_{112}u_{332}+2u_{221}u_{331}}{\sigma_{1}}\\
 &  & -\frac{(1+\epsilon)[\sum_{k\neq i}(\sigma_{2}^{ii}-\sigma_{2}^{kk})u_{kki}+f_{i}]^{2}}{\sigma_{1}^{2}\sigma_{2}^{ii}}\\
 &  & -C\sigma_{1}+\sum_{i}f_{u_{i}}b_{i}.
\end{eqnarray}
Then we can write explicitly of the second last term and use Cauchy-Schwarz
inequality and Lemma \ref{lemma2}, 
\begin{eqnarray}
-\frac{(1+\epsilon)[\sum_{k\neq i}(\sigma_{2}^{ii}-\sigma_{2}^{kk})u_{kki}+f_{i}]^{2}}{\sigma_{1}^{2}\sigma_{2}^{ii}} & \geq\\
-\frac{(1+2\epsilon)[(\lambda_{2}-\lambda_{1})u_{221}+(\lambda_{3}-\lambda_{1})u_{331}]^{2}}{\sigma_{1}^{2}\sigma_{2}^{11}}\\
-\frac{(1+2\epsilon)[(\lambda_{1}-\lambda_{2})u_{112}+(\lambda_{3}-\lambda_{2})u_{332}]^{2}}{\sigma_{1}^{2}\sigma_{2}^{22}}\\
-\frac{(1+2\epsilon)[(\lambda_{2}-\lambda_{3})u_{223}+(\lambda_{1}-\lambda_{3})u_{113}]^{2}}{\sigma_{1}^{2}\sigma_{2}^{33}}\\
-\frac{C}{\epsilon}\sigma_{1}.
\end{eqnarray}
 Due to symmetry, we only need to give the lower bound of the terms
which contain $u_{221}$ and $u_{331}$. We denote these terms by
$A_{1}$. 
\begin{eqnarray}
A_{1} & := & \frac{2u_{221}^{2}}{\sigma_{1}}+\frac{2u_{331}^{2}}{\sigma_{1}}-\frac{2(u_{221}+u_{331})f_{1}}{\sigma_{1}\sigma_{2}^{11}}\\
 &  & +\frac{2\sigma_{2}^{22}u_{221}^{2}}{\sigma_{1}\sigma_{2}^{11}}+\frac{2\sigma_{2}^{33}u_{331}^{2}}{\sigma_{1}\sigma_{2}^{11}}+\frac{2(\sigma_{2}^{22}+\sigma_{2}^{33})u_{221}u_{331}}{\sigma_{1}\sigma_{2}^{11}}\\
 &  & -\frac{2u_{221}u_{331}}{\sigma_{1}}-\frac{(1+2\epsilon)[(\lambda_{2}-\lambda_{1})u_{221}+(\lambda_{3}-\lambda_{1})u_{331}]^{2}}{\sigma_{1}^{2}\sigma_{2}^{11}}.
\end{eqnarray}
By Cauchy-Schwarz and Lemma \ref{lemma2} we have 
\begin{eqnarray}
-\frac{2(u_{221}+u_{331})f_{1}}{\sigma_{1}\sigma_{2}^{11}} & \geq & -\frac{2\epsilon^{2}\sigma_{1}(u_{221}+u_{331})^{2}}{\sigma_{1}\sigma_{2}^{11}}-\frac{f_{1}^{2}}{2\epsilon^{2}\sigma_{2}^{11}\sigma_{1}^{2}}\\
 & \geq & -\frac{2\epsilon^{2}\sigma_{1}(u_{221}+u_{331})^{2}}{\sigma_{1}\sigma_{2}^{11}}-\frac{C}{\epsilon^{2}}\sigma_{1}.
\end{eqnarray}
Then we get 
\begin{eqnarray}
A_{1} & \geq & \frac{2\sigma_{2}^{11}+2\sigma_{2}^{22}}{\sigma_{1}\sigma_{2}^{11}}u_{221}^{2}+\frac{2\sigma_{2}^{11}+2\sigma_{2}^{33}}{\sigma_{1}\sigma_{2}^{11}}u_{331}^{2}\\
 &  & +\frac{4\lambda_{1}}{\sigma_{1}\sigma_{2}^{11}}u_{221}u_{331}-\frac{2\epsilon^{2}\sigma_{1}(u_{221}+u_{331})^{2}}{\sigma_{1}\sigma_{2}^{11}}\\
 &  & -\frac{(1+2\epsilon)[(\lambda_{2}-\lambda_{1})u_{221}+(\lambda_{3}-\lambda_{1})u_{331}]^{2}}{\sigma_{1}^{2}\sigma_{2}^{11}}\\
 &  & -\frac{C}{\epsilon^{2}}\sigma_{1}.
\end{eqnarray}
We will prove that $A_{1}\geq-\frac{C}{\epsilon^{2}}\sigma_{1}$ in
the following elementary but tedious two claims \ref{claim1} and
\ref{claim2}. Then we choose $\epsilon=\frac{1}{100}$, such that
\begin{equation}
1+\delta\geq\frac{1+2\epsilon}{1-\epsilon},
\end{equation}
where $\delta$ is small constant in the Claim \ref{claim2}. 

In all, we will get 
\begin{equation}
A\geq-C\sigma_{1}+\sum_{i}f_{u_{i}}b_{i}.
\end{equation}
\end{proof}
\begin{claim}
\label{claim1} For fix $\epsilon\leq\frac{2}{5}$, we have 
\begin{equation}
\frac{2\sigma_{2}^{11}+2\sigma_{2}^{22}}{\sigma_{1}\sigma_{2}^{11}}u_{221}^{2}+\frac{2\sigma_{2}^{11}+2\sigma_{2}^{33}}{\sigma_{1}\sigma_{2}^{11}}u_{331}^{2}+\frac{4\lambda_{1}}{\sigma_{1}\sigma_{2}^{11}}u_{221}u_{331}\geq\frac{2\epsilon\sigma_{1}(u_{221}+u_{331})^{2}}{\sigma_{1}\sigma_{2}^{11}}.
\end{equation}
\end{claim}
\begin{proof}
This claim follows from elementary inequality 
\begin{eqnarray}
(\sigma_{2}^{11}+\sigma_{2}^{22}-\epsilon\sigma_{1})(\sigma_{2}^{11}+\sigma_{2}^{33}-\epsilon\sigma_{1})-(\lambda_{1}-\epsilon\sigma_{1})^{2}\\
=(1-\epsilon)^{2}\sigma_{1}^{2}+(1-\epsilon)\sigma_{1}(\lambda_{2}+\lambda_{3})+\lambda_{2}\lambda_{3}-(\lambda_{1}-\epsilon\sigma_{1})^{2}\\
\geq(1-\epsilon)^{2}(\lambda_{2}^{2}+\lambda_{3}^{2}+2f)+(1+2\epsilon)(1-\epsilon)(f-\lambda_{2}\lambda_{3})\\
+(1-\epsilon-\epsilon^{2})(\lambda_{2}+\lambda_{3})^{2}+\lambda_{2}\lambda_{3}\\
\geq(1-\epsilon)^{2}(\lambda_{2}^{2}+\lambda_{3}^{2})-\epsilon(1-2\epsilon)\lambda_{2}\lambda_{3}\\
\geq2(1-\epsilon)^{2}|\lambda_{2}\lambda_{3}|-\epsilon(1-2\epsilon)\lambda_{2}\lambda_{3}.
\end{eqnarray}

If we assume $2(1-\epsilon)^{2}-\epsilon(1-2\epsilon)\geq0$, we have
above inequality nonnegative.\end{proof}
\begin{claim}
\label{claim2} For fix $\delta\leq\frac{1}{20}$ , we have 
\begin{eqnarray}
\frac{2\sigma_{2}^{11}+2\sigma_{2}^{22}}{\sigma_{1}\sigma_{2}^{11}}u_{221}^{2}+\frac{2\sigma_{2}^{11}+2\sigma_{2}^{33}}{\sigma_{1}\sigma_{2}^{11}}u_{331}^{2}+\frac{4\lambda_{1}}{\sigma_{1}\sigma_{2}^{11}}u_{221}u_{331}\\
\geq\frac{(1+\delta)[(\lambda_{2}-\lambda_{1})u_{221}+(\lambda_{3}-\lambda_{1})u_{331}]^{2}}{\sigma_{1}^{2}\sigma_{2}^{11}}.
\end{eqnarray}
\end{claim}
\begin{proof}
In order to prove this claim, we need to prove 
\begin{eqnarray}
[2(\sigma_{2}^{11}+\sigma_{2}^{22})\sigma_{1}-(1+\delta)(\lambda_{1}-\lambda_{2})^{2}]\\
\times[2(\sigma_{2}^{11}+\sigma_{2}^{33})\sigma_{1}-(1+\delta)(\lambda_{1}-\lambda_{3})^{2}]\\
\geq[2\lambda_{1}\sigma_{1}-(1+\delta)(\lambda_{1}-\lambda_{2})(\lambda_{1}-\lambda_{3})]^{2}.\label{eq:elem}
\end{eqnarray}

Because we have 
\begin{eqnarray}
 & 2(\sigma_{2}^{11}+\sigma_{2}^{22})\sigma_{1}-(1+\delta)(\lambda_{1}-\lambda_{2})^{2}\\
= & 2\sigma_{1}^{2}+2\lambda_{3}\sigma_{1}-(1+\delta)(\lambda_{1}-\lambda_{2})^{2}\\
= & 2\lambda_{1}^{2}+2\lambda_{2}^{2}+2\lambda_{3}^{2}+4f+2\lambda_{3}^{2}+2\lambda_{3}\sigma_{2}^{33}-(1+\delta)(\lambda_{1}-\lambda_{2})^{2}\\
= & (1-\delta)\lambda_{1}^{2}+(1-\delta)\lambda_{2}^{2}+4\lambda_{3}^{2}+6f+2\delta\lambda_{1}\lambda_{2}.
\end{eqnarray}

And similarly 
\begin{eqnarray}
 & 2(\sigma_{2}^{11}+\sigma_{2}^{33})\sigma_{1}-(1+\delta)(\lambda_{1}-\lambda_{3})^{2}\\
= & (1-\delta)\lambda_{1}^{2}+(1-\delta)\lambda_{3}^{2}+4\lambda_{2}^{2}+6f+2\delta\lambda_{1}\lambda_{3}.
\end{eqnarray}

We also compute right hand side of (\ref{eq:elem}) 
\begin{eqnarray}
 & 2\lambda_{1}\sigma_{1}-(1+\delta)(\lambda_{1}-\lambda_{2})(\lambda_{1}-\lambda_{3})\\
= & 2\lambda_{1}^{2}+2\lambda_{1}\sigma_{2}^{11}-(1+\delta)(\lambda_{1}^{2}-\lambda_{2}\lambda_{1}-\lambda_{3}\lambda_{1}+\lambda_{2}\lambda_{3})\\
= & (1-\delta)\lambda_{1}^{2}+2f-2\lambda_{2}\lambda_{3}+(1+\delta)(f-2\lambda_{2}\lambda_{3})\\
= & (1-\delta)\lambda_{1}^{2}+(3+\delta)f-2(2+\delta)\lambda_{2}\lambda_{3}.
\end{eqnarray}

Then we have 
\begin{eqnarray}
[(1-\delta)\lambda_{1}^{2}+(1-\delta)\lambda_{2}^{2}+4\lambda_{3}^{2}+6f+2\delta\lambda_{1}\lambda_{2}]\\
\times[(1-\delta)\lambda_{1}^{2}+(1-\delta)\lambda_{3}^{2}+4\lambda_{2}^{2}+6f+2\delta\lambda_{1}\lambda_{3}]\\
-[(1-\delta)\lambda_{1}^{2}+(3+\delta)f-2(2+\delta)\lambda_{2}\lambda_{3}]^{2}\\
=[(1-\delta)\lambda_{1}^{2}+(3+\delta)f][6f+(5-\delta)(\lambda_{2}^{2}+\lambda_{3}^{2})-2\delta\lambda_{2}\lambda_{3}]\\
+[(3-\delta)f+(1-\delta)\lambda_{2}^{2}+4\lambda_{3}^{2}+2\delta\lambda_{1}\lambda_{2}]\\
\times[(3-\delta)f+(1-\delta)\lambda_{3}^{2}+4\lambda_{2}^{2}+2\delta\lambda_{1}\lambda_{3}]\\
+4[(1-\delta)\lambda_{1}^{2}+(3+\delta)f](2+\delta)\lambda_{2}\lambda_{3}-4(2+\delta)^{2}\lambda_{2}^{2}\lambda_{3}^{2}\\
=[(1-\delta)\lambda_{1}^{2}+(3+\delta)f][6f+(5-\delta)(\lambda_{2}^{2}+\lambda_{3}^{2})+2(4+\delta)\lambda_{2}\lambda_{3}]\\
+[(3-\delta)f+(1-\delta)\lambda_{2}^{2}+4\lambda_{3}^{2}][(3-\delta)f+(1-\delta)\lambda_{3}^{2}+4\lambda_{2}^{2}]\\
+2\delta\lambda_{1}\lambda_{2}[(3-\delta)f+(1-\delta)\lambda_{3}^{2}+4\lambda_{2}^{2}]\\
+2\delta\lambda_{1}\lambda_{3}[(3-\delta)f+(1-\delta)\lambda_{2}^{2}+4\lambda_{3}^{2}]\\
+4\delta^{2}\lambda_{1}^{2}\lambda_{2}\lambda_{3}-4(2+\delta)^{2}\lambda_{2}^{2}\lambda_{3}^{2}.\label{eq:ele0}
\end{eqnarray}
The point in the above computation is that term $[(1-\delta)\lambda_{1}^{2}+(3+\delta)f]^{2}$
cancels exactly. Moreover, other terms has a lot of room to play with
if you choose $\delta$ small.

We deal with these terms and assume that $\delta\leq\frac{1}{20}$
\begin{eqnarray}
[(1-\delta)\lambda_{1}^{2}+(3+\delta)f]\\
\times[6f+(5-\delta)(\lambda_{2}^{2}+\lambda_{3}^{2})+2(4+\delta)\lambda_{2}\lambda_{3}]\\
\geq(1-\delta)\lambda_{1}^{2}(1-2\delta)(\lambda_{2}^{2}+\lambda_{3}^{2})\\
\geq-4\delta^{2}\lambda_{1}^{2}\lambda_{2}\lambda_{3}.\label{eq:ele1}
\end{eqnarray}

And 
\begin{eqnarray}
2\delta\lambda_{1}\lambda_{2}[(3-\delta)f+(1-\delta)\lambda_{3}^{2}+4\lambda_{2}^{2}]\\
+2\delta\lambda_{1}\lambda_{3}[(3-\delta)f+(1-\delta)\lambda_{2}^{2}+4\lambda_{3}^{2}]\\
=2\delta(f-\lambda_{2}\lambda_{3})(3-\delta)f+2\delta(1-\delta)(f-\lambda_{2}\lambda_{3})(\lambda_{2}^{2}+\lambda_{3}^{2})\\
+2\delta(3+\delta)\lambda_{1}(\lambda_{2}^{3}+\lambda_{3}^{3})\\
\geq-2\delta(3-\delta)\lambda_{2}\lambda_{3}f-2\delta(1-\delta)\lambda_{2}\lambda_{3}(\lambda_{2}^{2}+\lambda_{3}^{2})\\
+2\delta(3+\delta)(f-\lambda_{2}\lambda_{3})(\lambda_{2}^{2}-\lambda_{2}\lambda_{3}+\lambda_{3}^{2})\\
\geq-2\delta(3-\delta)\lambda_{2}\lambda_{3}f\\
-2\delta\lambda_{2}\lambda_{3}(4\lambda_{2}^{2}+4\lambda_{3}^{2}-(3+\delta)\lambda_{2}\lambda_{3}).\label{eq:ele2}
\end{eqnarray}

We also have, 
\begin{eqnarray}
[(3-\delta)f+(1-\delta)\lambda_{2}^{2}+4\lambda_{3}^{2}][(3-\delta)f+(1-\delta)\lambda_{3}^{2}+4\lambda_{2}^{2}]\\
\geq[16+(1-\delta)^{2}]\lambda_{2}^{2}\lambda_{3}^{2}+(3-\delta)f(5-\delta)(\lambda_{2}^{2}+\lambda_{3}^{2})\\
+4(1-\delta)(\lambda_{2}^{4}+\lambda_{3}^{4})\\
\geq4(2+\delta)^{2}\lambda_{2}^{2}\lambda_{3}^{2}+16f|\lambda_{2}\lambda_{3}|+6\delta(\lambda_{2}^{2}+\lambda_{3}^{2})^{2}\\
\geq4(2+\delta)^{2}\lambda_{2}^{2}\lambda_{3}^{2}+2\delta(3-\delta)\lambda_{2}\lambda_{3}f\\
+2\delta\lambda_{2}\lambda_{3}(4\lambda_{2}^{2}+4\lambda_{3}^{2}-(3+\delta)\lambda_{2}\lambda_{3}).\label{eq:ele3}
\end{eqnarray}

We combine inequalities (\ref{eq:ele1}), (\ref{eq:ele2}) and (\ref{eq:ele3})
into (\ref{eq:ele0}) to verify inequality (\ref{eq:elem}). Finally
we complete the proof of this claim.\end{proof}
\begin{lem}
\label{lem4}Suppose $u$ are admissible solutions of the equations
(\ref{eq:sigma2}) on $B_{10}\subset\mathbb{R}^{3}$ , then we have
\begin{equation}
\sup_{B_{5}}\sigma_{1}\leq C\sup_{B_{1}}\sigma_{1}.\label{eq:est}
\end{equation}
\end{lem}
\begin{proof}
Denote $\rho(x)=10^{2}-|x|^{2}$, and $M_{r}:=\sup_{B_{1}}\sigma_{1}$.
We consider a test function in $B_{10}$ 
\begin{equation}
P(x)=2\log\rho(x)+g(x\cdot Du-u)+\log\log\max\{\frac{\sigma_{1}}{M_{1}},20\},\label{eq:test}
\end{equation}
where $g(t)=-\frac{1}{2\beta}(\log(1-\frac{t}{20\max|Du|+2\max|u|+1}))$
and $\beta$ is a larger number to be determined later. 

Assume $P$ attains its maximum point, say $x_{0}\in B_{10}$, and
$\frac{\sigma_{1}}{\sup_{B_{1}}\sigma_{1}}>20$. Then $P$ must attained
its maximum point in the ring $B_{10}/B_{1}$. Moreover, we may always
assume $\frac{\sigma_{1}}{\sup_{B_{1}}\sigma_{1}}$ sufficiant large.
We choose coordinate frame $\{e_{1},e_{2},e_{3}\}$, such that $D^{2}u$
is diagonalized at this point. By maximum principle, at $x_{0}$,
we have
\begin{equation}
0=P_{i}=\frac{2\rho_{i}}{\rho}+g^{\prime}x_{k}u_{ki}+\frac{b_{i}}{b-\log M_{1}}.\label{eq:fp}
\end{equation}
and
\begin{eqnarray}
P_{ij} & = & \frac{2\rho_{ij}}{\rho}-\frac{2\rho_{i}\rho_{j}}{\rho^{2}}\\
 &  & +g^{\prime\prime}x_{k}u_{ki}x_{l}u_{lj}+g^{\prime}(u_{ij}+x_{k}u_{kij})\\
 &  & +\frac{b_{ij}}{b-\log M_{1}}-\frac{b_{i}b_{j}}{(b-\log M_{1})^{2}}.
\end{eqnarray}
Contracting with $\sigma_{2}^{ij}:=\frac{\partial\sigma_{2}(D^{2}u)}{\partial u_{ij}}$,
we get 
\begin{eqnarray}
\sigma_{2}^{ij}P_{ij} & = & -4\frac{\sum\sigma_{2}^{ii}}{\rho}-8\frac{\sigma_{2}^{ii}x_{i}^{2}}{\rho^{2}}\nonumber \\
 &  & +g^{\prime\prime}\sigma_{2}^{ii}(x_{i}u_{ii})^{2}+2g^{\prime}f+g^{\prime}x_{k}f_{k}\nonumber \\
 &  & +\frac{\sigma_{2}^{ii}b_{ii}}{b-\log M_{1}}-\frac{\sigma_{2}^{ii}b_{i}^{2}}{(b-\log M_{1})^{2}}.\label{eq:P1}
\end{eqnarray}
Using Lemma \ref{lem1}, we get 
\begin{eqnarray}
\sigma_{2}^{ij}P_{ij} & \geq & -4\frac{\sum\sigma_{2}^{ii}}{\rho}-8\frac{\sigma_{2}^{ii}x_{i}^{2}}{\rho^{2}}\\
 &  & +g^{\prime\prime}\sigma_{2}^{ii}(x_{i}u_{ii})^{2}+2g^{\prime}f+g^{\prime}x_{k}f_{k}\\
 &  & +\frac{\sigma_{2}^{ii}b_{i}^{2}}{100(b-\log M_{1})}+\frac{\sum f_{u_{i}}b_{i}-C\sigma_{1}}{b-\log M_{1}}.
\end{eqnarray}
By (\ref{eq:fp}), 
\begin{equation}
g^{\prime}x_{k}f_{k}+\frac{\sum_{i}f_{u_{i}}b_{i}}{b-\log M_{1}}\geq-C(g^{\prime}+\frac{1}{\rho}).
\end{equation}
We may assume $\rho^{2}(x_{0})\log\frac{\sigma_{1}(x_{0})}{M_{1}}\geq C$
. We also use Newton -MacLaur in inequality $\sigma_{1}\sigma_{2}\geq9\sigma_{3}$
to get 
\begin{eqnarray}
\sigma_{2}^{ij}P_{ij} & \geq & -4\frac{\sum\sigma_{2}^{ii}}{\rho}-8\frac{\sigma_{2}^{ii}x_{i}^{2}}{\rho^{2}}+g^{\prime\prime}\sigma_{2}^{ii}(x_{i}u_{ii})^{2}\\
 &  & +\frac{\sigma_{2}^{ii}b_{i}^{2}}{100(b-\log M_{1})}-\frac{C\sigma_{1}}{b-\log M_{1}}.\label{eq:P2-1}
\end{eqnarray}
Then we divided into two cases.

Case 1: $x_{1}^{2}\geq\frac{1}{6}$.

From (\ref{eq:fp}), we know 
\begin{eqnarray*}
0 & = & -\frac{4x_{1}}{\rho}+g^{\prime}x_{1}u_{11}+\frac{b_{1}}{b-\log M_{1}}.
\end{eqnarray*}
We can assume 
\begin{eqnarray*}
g^{\prime}|x_{1}|u_{11} & \geq & \frac{12|x_{1}|}{\rho}.
\end{eqnarray*}
Otherwise we get the estimate from
\begin{equation}
g^{\prime}(b-\log M_{1})\leq g^{\prime}u_{11}\leq\frac{12}{\rho}.
\end{equation}
Then we have 
\begin{equation}
(\frac{b_{1}}{b-\log M_{1}})^{2}\geq\frac{(g^{\prime})^{2}x_{1}^{2}u_{11}^{2}}{3}\geq\frac{(g^{\prime})^{2}u_{11}^{2}}{18}.
\end{equation}
Inserting this inequality into (\ref{eq:P2-1}), we get from Lemma
\ref{lemma2}
\begin{eqnarray}
\sigma_{2}^{ij}P_{ij} & \geq & -4\frac{\sum\sigma_{2}^{ii}}{\rho}-8\frac{\sigma_{2}^{ii}x_{i}^{2}}{\rho^{2}}+\frac{\sigma_{2}^{ii}b_{i}^{2}}{100(b-\log M_{1})}-\frac{C\sigma_{1}}{b-\log M_{1}}.\\
 & \geq & -4\frac{\sum\sigma_{2}^{ii}}{\rho}-800\frac{\sum\sigma_{2}^{ii}}{\rho^{2}}+\frac{(g^{\prime})^{2}\sigma_{2}^{11}u_{11}^{2}(b-\log M_{1})}{180}\\
 & \geq & -4\frac{\sum\sigma_{2}^{ii}}{\rho}-800\frac{\sum\sigma_{2}^{ii}}{\rho^{2}}+\frac{c_{1}(g^{\prime})^{2}(b-\log M_{1})\sum\sigma_{2}^{ii}}{180}.
\end{eqnarray}
So we obtain the estimate
\begin{equation}
\rho^{2}(x)\log\frac{\sigma_{1}(x)}{M_{1}}\leq C\rho^{2}(x_{0})\log\frac{\sigma_{1}(x_{0})}{M_{1}}\leq C.
\end{equation}
Case 2: there exists $j\geq2$, such that $x_{j}^{2}\geq\frac{1}{6}$.

Using (\ref{eq:fp}), we have 
\begin{eqnarray*}
\beta\frac{\sigma_{2}^{jj}b_{j}^{2}}{(b-\log M_{1})^{2}} & = & \beta\sigma_{2}^{jj}[-\frac{4x_{j}}{\rho}+g^{\prime}x_{j}u_{jj}]^{2}\\
 & \geq & 2\beta\frac{\sigma_{2}^{jj}x_{j}^{2}}{\rho^{2}}-2\beta\sigma_{2}^{jj}(g^{\prime})^{2}(x_{j}u_{jj})^{2}.
\end{eqnarray*}
Then we get 
\begin{eqnarray*}
\sigma_{2}^{ij}P_{ij} & \geq & -4\frac{\sum\sigma_{2}^{ii}}{\rho}+(2\beta-8)\frac{\sigma_{2}^{jj}x_{j}^{2}}{\rho^{2}}-\frac{C\sigma_{1}}{b-\log M_{1}}\\
 &  & +(g^{\prime\prime}-2\beta(g^{\prime})^{2})\sigma_{2}^{ii}(x_{i}u_{ii})^{2}.
\end{eqnarray*}
If we choose $g^{\prime\prime}-2\beta(g^{\prime})^{2}>0$, by Lemma
\ref{lemma2} we have 
\begin{eqnarray*}
\sigma_{2}^{ij}P_{ij} & \geq & -4\frac{\sum\sigma_{2}^{ii}}{\rho}+(2\beta-8)\frac{\sigma_{2}^{jj}}{6\rho^{2}}\\
 & \geq & \frac{\sum\sigma_{2}^{ii}}{\rho}(-4+\frac{(\beta-2)c_{2}}{300}).
\end{eqnarray*}
We choose constant above by the following order, first $\beta=\frac{1200}{c_{2}}+3$,
then $g(t)=-\frac{1}{2\beta}(\log(1-\frac{t}{20\max|Du|+2\max|u|+1}))$
such that $g^{\prime\prime}\geq2\beta(g^{\prime})^{2}$. We finally
get the estimate (\ref{eq:est}).
\end{proof}
Similarly, if we consider quantity $V:=\log(f\sigma_{1}-\sigma_{3})$
we also have the following Lemma.
\begin{lem}
\label{lem5}Suppose $u$ are admissible solutions of equations (\ref{eq:sigma2})
on $B_{10}\subset\mathbb{R}^{3}$ , then we have in $B_{4}$ that 

\begin{equation}
\sigma_{2}^{ij}V{}_{ij}\geq-C\sigma_{1}+f_{u_{i}}V{}_{i},\label{eq:subV}
\end{equation}
and
\begin{equation}
\sigma_{1}f-\sigma_{3}\leq\frac{M_{5}}{\log^{\frac{1}{2}}M_{5}}\log(\sigma_{1}f-\sigma_{3}).\label{eq:V-1}
\end{equation}
\end{lem}
\begin{proof}
For the proof of the first inequality (\ref{eq:subV}) which is tedious
but similar to Lemma \ref{lem3}, we give its details in the appendix.
The second inequality (\ref{eq:V-1}) follows from (\ref{eq:subV})
almost the same as the proof of Lemma \ref{lem4}. The proof of the
second inequality will also be in the appendix. 
\end{proof}

\section{Mean value inequality.}

In this section we prove a mean value type inequality
\begin{thm}
\label{thm-meanvalue}Suppose $u$ are admissible solutions of equations
(\ref{eq:sigma2}) on $B_{10}\subset\mathbb{R}^{3}$ , then we have
\end{thm}
\begin{equation}
\sup_{B_{1}}\sigma_{1}=\sigma_{1}(y_{0})\leq C\int_{B_{1}(y_{0})}\sigma_{1}(x)(\sigma_{1}f-\sigma_{3})dx.\label{eq:meanvalue}
\end{equation}

We prove this theorem similar as Michael-Simon \cite{michael1973sobolev}. 
\begin{proof}
First we know from concavety of $\sigma_{2}^{\frac{1}{2}}$, we have
\begin{eqnarray}
\sigma_{2}^{ij}(\sigma_{1}){}_{ij} & \geq & \triangle f-\frac{|Df|^{2}}{2f}\\
 & \geq & -C(\triangle u)^{2}+f_{u_{i}}(\triangle u)_{i}.
\end{eqnarray}
Let $\chi$ be a non-negative, non-decreasing function in $C^{1}(\mathbb{R})$
with support in the interval $(0,\infty)$ and set 

\begin{equation}
\psi(r)=\int_{r}^{\infty}t\chi(\rho-t)dt
\end{equation}

where $0<\rho<10,$ $r^{2}=f(x,u(x),Du(x))|x-y_{0}|^{2}+|Du(x)-Du(y_{0})|^{2}$.
Let us denote $\mathfrak{B}_{r}=\{x\in B_{10}(y_{0})|f(x,u,Du)|x-y_{0}|^{2}+|Du(x)-Du(y_{0})|^{2}\leq r^{2}\}$.
We may assume that $\bar{y}_{0}:=(y_{0},Du(y_{0}))=(0,0)$.

By direct computation, we have
\begin{equation}
2rr_{i}=f_{i}|x|^{2}+2fx_{i}+2u_{k}u_{kj},\label{eq:r1}
\end{equation}

and 
\begin{equation}
2r_{i}r_{j}+2rr_{ij}=f_{ij}|x|^{2}+2f_{i}x_{j}+2f_{j}x_{i}+2f\delta_{ij}+2u_{ki}u_{kj}+2u_{k}u_{kij}.\label{eq:r2}
\end{equation}

Because $\lambda_{2}+\lambda_{3}>0$, we may assume $\lambda_{3}<0$,
$\lambda_{1}\geq\lambda_{2}\geq\lambda_{3}$ and $\lambda_{1}$ large,
\begin{eqnarray}
f\sigma_{1}-\sigma_{3} & \geq & f\lambda_{1}+\lambda_{1}\lambda_{3}^{2}\\
 & \geq & -c\lambda_{1}\lambda_{3}\\
 & \geq & -c\lambda_{2}\lambda_{3}.\label{eq:lambda1}
\end{eqnarray}

By equation we also have 
\begin{equation}
\lambda_{2}\lambda_{1}=f-\lambda_{1}\lambda_{3}-\lambda_{2}\lambda_{3}\leq C(f\sigma_{1}-\sigma_{3}).\label{eq:lambda2}
\end{equation}

We then have from (\ref{eq:r1}), (\ref{eq:r2}), (\ref{eq:lambda1}),
(\ref{eq:lambda2}) and Lemma \ref{lem1} 
\begin{eqnarray}
\sigma_{2}^{ij}\psi_{ij} & = & \sigma_{2}^{ij}(-r_{i}r\chi(\rho-r))_{j}\\
 & = & -\sigma_{2}^{ij}r_{ij}r\chi(\rho-r)-\sigma_{2}^{ij}r_{i}r_{j}\chi(\rho-r)+\sigma_{2}^{ij}r_{i}r_{j}r\chi^{\prime}(\rho-r)\\
 & = & -\chi(\rho-r)\sigma_{2}^{ij}(f\delta_{ij}+u_{ki}u_{kj}+\frac{f_{ij}|x|^{2}}{2}+2f_{i}x_{j})\\
 &  & -\chi(\rho-r)\sigma_{2}^{ij}u_{k}u_{kij}+\sigma_{2}^{ij}r_{i}r_{j}r\chi^{\prime}(\rho-r)\\
 & \leq & -3(\sigma_{1}f-\sigma_{3})\chi+C(r^{2}\chi+r\chi)(f\sigma_{1}-\sigma_{3})+\sigma_{2}^{ij}r_{i}r_{j}r\chi^{\prime}.\label{eq:psi}
\end{eqnarray}

We claim that 
\begin{equation}
\sigma_{2}^{ij}r_{i}r_{j}\leq(f\sigma_{1}-\sigma_{3})(1+Cr).\label{eq:cla}
\end{equation}

In fact, 
\begin{eqnarray}
\frac{\sigma_{2}^{ij}(f_{i}|x|^{2}+2fx_{i}+2u_{k}u_{kj})(f_{j}|x|^{2}+2fx_{j}+2u_{l}u_{lj})}{4r^{2}}\\
\leq C(f\sigma_{1}-\sigma_{3})r+\frac{\sigma_{2}^{ii}(fx_{i}+u_{i}u_{ii})^{2}}{r^{2}}.
\end{eqnarray}

Moreover, we have following elementary properties
\begin{equation}
(f\sigma_{1}-\sigma_{3})\delta_{ij}-f\sigma_{2}^{ij}=\sigma_{2}^{kl}u_{ki}u_{lj}.
\end{equation}

and 
\begin{equation}
f\sigma_{2}^{ii}u_{ii}^{2}x_{i}^{2}-2f\sigma_{2}^{ii}x_{i}u_{i}u_{ii}+f\sigma_{2}^{ii}u_{i}^{2}\geq f\sigma_{2}^{ii}(u_{ii}x_{i}-u_{i})^{2}.
\end{equation}

Then we have 
\begin{equation}
\frac{\sigma_{2}^{ii}(fx_{i}+u_{i}u_{ii})^{2}}{r^{2}}\leq(f\sigma_{1}-\sigma_{3}).
\end{equation}

We obtain from (\ref{eq:cla}) and (\ref{eq:psi}) that 
\begin{equation}
\sigma_{2}^{ij}\psi_{ij}\leq(\sigma_{1}f-\sigma_{3})[-3\chi+C(r^{2}\chi+r\chi)+(1+r)r\chi^{\prime}].
\end{equation}

Then mutiply both side by $\sigma_{1}$ and integral on the domain
$\mathfrak{B}_{10}$,
\begin{eqnarray}
\int_{\mathfrak{B}_{10}}\sigma_{1}\sigma_{2}^{ij}\psi_{ij}dx & \leq & \rho^{4}\frac{d}{d\rho}(\int_{\mathfrak{B}_{10}}\frac{\sigma_{1}\chi(\rho-r)}{\rho^{3}}(\sigma_{1}f-\sigma_{3})dx)\label{eq:mean}\\
 &  & +C\int_{\mathfrak{B}_{10}}\sigma_{1}r^{2}\chi^{\prime}(\sigma_{1}f-\sigma_{3})dx\\
 &  & +C\int_{\mathfrak{B}_{10}}r\sigma_{1}\chi(\rho-r)(\sigma_{1}f-\sigma_{3})dx.
\end{eqnarray}

By (\ref{eq:logb}), we have 
\begin{equation}
-C\int_{\mathfrak{B}_{10}}\sigma_{1}^{2}\psi dx+\int_{\mathfrak{B}_{10}}f_{u_{i}}(\sigma_{1})_{i}\mbox{\ensuremath{\psi}}\leq\int_{\mathfrak{B}_{10}}\sigma_{1}\sigma_{2}^{ij}\psi_{ij}dx.\label{eq:step3}
\end{equation}

Inserting (\ref{eq:step3}) into (\ref{eq:mean}), we get 
\begin{eqnarray}
-\frac{d}{d\rho}(\int_{\mathfrak{B}_{10}}\frac{\sigma_{1}\chi(\rho-r)}{\rho^{3}}(\sigma_{1}f-\sigma_{3})dx) & \leq & \frac{C\int_{\mathfrak{B}_{10}}r\sigma_{1}\chi(\rho-r)(\sigma_{1}f-\sigma_{3})dx}{\rho^{4}}.\\
 & + & \frac{C\int_{\mathfrak{B}_{10}}\sigma_{1}r^{2}\chi^{\prime}(\sigma_{1}f-\sigma_{3})dx}{\rho^{4}}\\
 & + & \frac{C\int_{\mathfrak{B}_{10}}\sigma_{1}^{2}\psi dx-\int_{\mathfrak{B}_{10}}f_{u_{i}}(\sigma_{1})_{i}\psi dx}{\rho^{4}}.
\end{eqnarray}

Because $\chi$, $\chi^{\prime}$ and $\psi$ are all supported in
$\mathfrak{B}_{\rho}$, we deal with right hand side of above inequality
term by term. For the first term, we have 
\begin{equation}
\frac{C\int_{\mathfrak{B}_{10}}r\sigma_{1}\chi(\rho-r)(\sigma_{1}f-\sigma_{3})dx}{\rho^{4}}\leq C\frac{\int_{\mathfrak{B}_{10}}\sigma_{1}\chi(\rho-r)(\sigma_{1}f-\sigma_{3})dx}{\rho^{3}}.\label{eq:term1}
\end{equation}

Then for the second term, we integral from $\delta$ to $R$, 

\begin{eqnarray}
\int_{\delta}^{R}\frac{\int_{\mathfrak{B}_{10}}\sigma_{1}r^{2}\chi^{\prime}(\sigma_{1}f-\sigma_{3})dx}{\rho^{4}}d\rho\\
\leq\int_{\delta}^{R}\frac{\int_{\mathfrak{B}_{10}}\sigma_{1}\chi^{\prime}(\sigma_{1}f-\sigma_{3})dx}{\rho^{2}}d\rho\\
\leq\frac{\int_{\mathfrak{B}_{10}}\sigma_{1}\chi(\sigma_{1}f-\sigma_{3})dx}{\rho^{2}}|_{\delta}^{R}\\
+\int_{\delta}^{R}\frac{2\int_{\mathfrak{B}_{10}}\sigma_{1}\chi(\sigma_{1}f-\sigma_{3})dx}{\rho^{3}}d\rho.\label{eq:term2}
\end{eqnarray}

For the last term, we use (\ref{eq:r1}) and definition of $\psi$
to estimate 
\begin{eqnarray}
\frac{C\int_{\mathfrak{B}_{10}}\sigma_{1}^{2}\psi dx-\int_{\mathfrak{B}_{10}}f_{u_{i}}(\sigma_{1})_{i}\psi dx}{\rho^{4}} & \leq & \frac{C\int_{\mathfrak{B}_{10}}\sigma_{1}^{2}\psi dx-\int_{\mathfrak{B}_{10}}f_{u_{i}}r_{i}\sigma_{1}r\chi dx}{\rho^{4}}\\
 & \leq & \frac{C\int_{\mathfrak{B}_{10}}\sigma_{1}^{2}\psi dx+\int_{\mathfrak{B}_{10}}\sigma_{1}^{2}r\chi dx}{\rho^{4}}\\
 & \leq & \frac{C\int_{\mathfrak{B}_{10}}\sigma_{1}\chi(\rho-r)(\sigma_{1}f-\sigma_{3})dx}{\rho^{3}}.\label{eq:term3}
\end{eqnarray}

Then we combine (\ref{eq:term1}), (\ref{eq:term2}) and (\ref{eq:term3})
with integrating from $0\leq\delta$ to $R\leq10$, 
\begin{equation}
\int_{\mathfrak{B}_{10}}\frac{\sigma_{1}(\sigma_{1}f-\sigma_{3})\chi(\delta-r)}{\delta^{3}}dx\leq C\int_{\mathfrak{B}_{10}}\frac{\sigma_{1}(\sigma_{1}f-\sigma_{3})\chi(R-r)}{R^{3}}dx
\end{equation}

Letting $\chi$ approximate the characteristic function of the interval
$(-\infty,0)$, in an appropriate fashion, we obtain, 
\begin{equation}
\frac{\int_{\mathfrak{B}_{\delta}}\sigma_{1}(\sigma_{1}f-\sigma_{3})}{\delta^{3}}\leq C\frac{\int_{\mathfrak{B}_{R}}\sigma_{1}(\sigma_{1}f-\sigma_{3})}{R^{3}}.\label{eq:monotonicity}
\end{equation}

Because the graph $(x,Du)$ where $u$ satisfied equation (\ref{eq:sigma2})
can be viewed as a three dimensional smooth submanifold in $(\mathbb{R}^{3}\times\mathbb{R}^{3},fdx^{2}+dy^{2})$
with volume form exactly $\sigma_{1}f-\sigma_{3}$. We note that the
boundedness of the mean curvature is encoded in the above proof. Moreover,
the geodesic ball with radius $\delta$ of this submanifold is comparable
with $\mathfrak{B}_{\delta}$. Then let $\delta\rightarrow0$, we
finally get 
\begin{eqnarray*}
\sigma_{1}(y_{0}) & \leq & C\frac{\int_{\mathfrak{B}_{R}(\bar{y}_{0})}\sigma_{1}(\sigma_{1}f-\sigma_{3})dx}{R^{3}}\leq C\frac{\int_{B_{R}(y_{0})}\sigma_{1}(\sigma_{1}f-\sigma_{3})dx}{R^{3}}.
\end{eqnarray*}

\end{proof}

\section{Proof of the theroem}
\begin{proof}
From Theorem \ref{thm-meanvalue}, Lemma \ref{lem5} and Lemma (\ref{lem4}),
we have at maximum point $x_{0}$ of $\bar{B}_{1}(0)$ 
\begin{eqnarray}
M_{1} & \leq & \int_{B_{1}(x_{0})}\sigma_{1}(\sigma_{1}f-\sigma_{3})dx\\
 & \leq & \frac{M_{5}}{\log^{\frac{1}{2}}M_{5}}\int_{B_{1}(x_{0})}\log(\sigma_{1}f-\sigma_{3})\sigma_{1}dx\\
 & \leq & \frac{CM_{1}}{\log^{\frac{1}{2}}M_{1}}\int_{B_{1}(x_{0})}\sigma_{1}\log\sigma_{1}dx.\label{eq:WY1}
\end{eqnarray}
Recall that 
\begin{equation}
\sigma_{2}^{ij}b{}_{ij}\geq\frac{1}{100}\sigma_{2}^{ij}b{}_{i}b{}_{j}-C\sigma_{1}+\sum_{i}f_{u_{i}}b_{i},\label{eq:logb-1}
\end{equation}
We have integral version of this inequality for any $r<5$,
\begin{equation}
\int_{B_{r}}-\sigma_{2}^{ij}\phi_{i}b_{i}\geq c_{0}\int_{B_{r}}\phi\sigma_{2}^{ij}b_{i}b_{j}-C(1+\int_{B_{r}}\sum_{i}f_{u_{i}}b_{i}\phi).\label{eq:intlog}
\end{equation}
for all non-negative $\phi\in C_{0}^{\infty}$. We choose different
cutoff functions all are denoted $0\le\phi\leq1$, which support in
larger ball $B_{r+1}(x_{0})$ and equals to $1$ in smaller ball $B_{r}(x_{0})$
and $|\nabla\phi|+|\nabla^{2}\phi|\leq C$. 
\begin{eqnarray}
\int_{B_{1}(x_{0})}b\sigma_{1} & \leq & \int_{B_{2}(x_{0})}\phi b\sigma_{1}\\
 & \leq & C(\int_{B_{2}(x_{0})}b+\int_{B_{2}(x_{0})}|\nabla b|)\\
 & \leq & C(1+\int_{B_{2}(x_{0})}|\nabla b|).\label{eq:WY2}
\end{eqnarray}
We only need to estimate $\int_{B_{2}(x_{0})}|\nabla b|$. We use
\begin{equation}
\sigma_{1}\sigma_{2}^{ij}\geq c\delta_{ij},
\end{equation}
to get 
\begin{equation}
\int_{B_{2}(x_{0})}|\nabla b|\leq\int_{B_{2}(x_{0})}\sqrt{\sigma_{2}^{ij}b_{i}b_{j}}\sqrt{\sigma_{1}}.
\end{equation}
By Holder inequality, we use (\ref{eq:logb-1}), 
\begin{eqnarray}
\int_{B_{2}(x_{0})}|\nabla b| & \leq & (\int_{B_{2}(x_{0})}\sigma_{2}^{ij}b_{i}b_{j})^{\frac{1}{2}}(\int_{B_{2}(x_{0})}\sigma_{1})^{\frac{1}{2}}\\
 & \leq & C\int_{B_{3}(x_{0})}\phi^{2}\sigma_{2}^{ij}b_{i}b_{j}.\label{eq:WY3}
\end{eqnarray}
Then using (\ref{eq:intlog}), we get 
\begin{eqnarray}
\int_{B_{3}(x_{0})}\phi^{2}\sigma_{2}^{ij}b_{i}b_{j} & \leq & C(\int_{B_{3}(x_{0})}\phi^{2}\sigma_{2}^{ij}b_{ij}+1+\int_{B_{3}(x_{0})}\phi^{2}|\nabla b|)\\
 & \leq & C(\int_{B_{3}(x_{0})}-\phi\sigma_{2}^{ij}\phi_{i}b_{i}+1\\
 &  & +\int_{B_{3}(x_{0})}\phi^{2}\sqrt{\sigma_{2}^{ij}b_{i}b_{j}}\sqrt{\sigma_{1}})
\end{eqnarray}
By Cauchy-Schwarz inequality
\begin{eqnarray}
\int_{B_{3}(x_{0})}\phi^{2}\sigma_{2}^{ij}b_{i}b_{j} & \leq & C(\epsilon\int_{B_{3}(x_{0})}\phi^{2}\sigma_{2}^{ij}b_{i}b_{j}+\int_{B_{3}(x_{0})}\sigma_{2}^{ij}\phi_{i}\phi_{i}+\frac{1}{\epsilon})\\
 & \leq & C\epsilon\int_{B_{3}(x_{0})}\phi^{2}\sigma_{2}^{ij}b_{i}b_{j}+\frac{C}{\epsilon}.
\end{eqnarray}
We Choose $\epsilon$ small such that $C\epsilon\leq\frac{1}{2}$,
\begin{eqnarray}
\int_{B_{3}(x_{0})}\phi^{2}\sigma_{2}^{ij}b_{i}b_{j} & \leq & C.\label{eq:WY4}
\end{eqnarray}
Finally, combine (\ref{eq:WY1}), (\ref{eq:WY2}), (\ref{eq:WY3})
and (\ref{eq:WY4}), we get the estimate 
\begin{equation}
\log M_{1}\leq C.
\end{equation}

\end{proof}

\section{Appendix}

We prove a differential Equation of $V=\log(f\sigma_{1}-\sigma_{3})$
in Lemma \ref{lem5}.
\begin{proof}
As in Lemma \ref{lem3}, we assume that $\{u_{ij}\}$ is diagonal.
The differential equations of $V$ are
\begin{equation}
V{}_{i}=\frac{(\sigma_{1})_{i}f+\sigma_{1}f_{i}-(\sigma_{3})_{i}}{\sigma_{1}f-\sigma_{3}},
\end{equation}
and
\begin{eqnarray}
V{}_{ij} & = & \frac{(\sigma_{1})_{ij}f+(\sigma_{1})_{i}f_{j}+(\sigma_{1})_{j}f_{i}+\sigma_{1}f_{ij}-(\sigma_{3})_{ij}}{\sigma_{1}f-\sigma_{3}}\\
 &  & -\frac{[(\sigma_{1})_{i}f+\sigma_{1}f_{i}-(\sigma_{3})_{i}][(\sigma_{1})_{j}f+\sigma_{1}f_{j}-(\sigma_{3})_{j}]}{(\sigma_{1}f-\sigma_{3})^{2}}.
\end{eqnarray}
We contract with $\sigma_{2}^{ij}$ and use (\ref{eq:f2}), 
\begin{eqnarray}
\sigma_{2}^{ij}V{}_{ij} & = & \frac{f\sigma_{2}^{ij}u_{kkij}+2\sigma_{2}^{ii}(\sigma_{1})_{i}f_{i}+\sigma_{1}\sigma_{2}^{ij}f_{ij}-\sigma_{2}^{ij}(\sigma_{3})_{ij}}{\sigma_{1}f-\sigma_{3}}\\
 &  & -\frac{\sigma_{2}^{ii}[(\sigma_{1})_{i}f+\sigma_{1}f_{i}-(\sigma_{3})_{i}]^{2}}{(\sigma_{1}f-\sigma_{3})^{2}}\\
 & = & \frac{-f\sum\limits _{i}\sum\limits _{k\neq p}u_{kki}u_{ppi}+f\sum\limits _{i}\sum\limits _{k\neq p}u_{kpi}^{2}+2\sigma_{2}^{ii}(\sigma_{1})_{i}f_{i}}{\sigma_{1}f-\sigma_{3}}\\
 &  & \frac{-\sum\limits _{k\neq p}\sigma_{2}^{ii}\sigma_{1}(\lambda|kp)u_{kki}u_{ppi}+\sum\limits _{k\neq p}\sigma_{2}^{ii}\sigma_{1}(\lambda|kp)u_{kpi}^{2}}{\sigma_{1}f-\sigma_{3}}\\
 &  & +\frac{\sigma_{1}\sigma_{2}^{ij}f_{ij}}{\sigma_{1}f-\sigma_{3}}-\frac{\sigma_{3}^{ii}(\sum\limits _{k\neq p}u_{kpi}^{2}-\sum\limits _{k\neq p}u_{kki}u_{ppi})}{\sigma_{1}f-\sigma_{3}}\\
 &  & -\frac{\sigma_{2}^{ii}[(\sigma_{1})_{i}f+\sigma_{1}f_{i}-(\sigma_{3})_{i}]^{2}}{(\sigma_{1}f-\sigma_{3})^{2}}+\frac{f\triangle f-\sigma_{3}^{ij}f_{ij}}{\sigma_{1}f-\sigma_{3}}.
\end{eqnarray}
We devided above equation into following four parts, 
\begin{eqnarray}
I: & = & \frac{\sum\limits _{i}\sum\limits _{k\neq p}(-f-\sigma_{2}^{ii}\sigma_{1}(\lambda|kp)+\sigma_{3}^{ii})u_{kki}u_{ppi}}{\sigma_{1}f-\sigma_{3}}\\
 & = & \frac{-\sum\limits _{i}\sum\limits _{k\neq p}\sigma_{2}^{ii}(\lambda_{i}+\sigma_{1}(\lambda|kp))u_{kki}u_{ppi}}{\sigma_{1}f-\sigma_{3}}.
\end{eqnarray}
\begin{eqnarray}
II: & = & \frac{\sum\limits _{i}\sum\limits _{k\neq p}(f+\sigma_{2}^{ii}\sigma_{1}(\lambda|kp)-\sigma_{3}^{ii})u_{kpi}^{2}}{\sigma_{1}f-\sigma_{3}}\\
 & = & \frac{\sum\limits _{i}\sum\limits _{k\neq p}\sigma_{2}^{ii}(\lambda_{i}+\sigma_{1}(\lambda|kp))u_{kpi}^{2}}{\sigma_{1}f-\sigma_{3}}.
\end{eqnarray}
\begin{eqnarray}
III: & = & -\frac{\sigma_{2}^{ii}[(\sigma_{1})_{i}f+\sigma_{1}f_{i}-(\sigma_{3})_{i}]^{2}}{(\sigma_{1}f-\sigma_{3})^{2}}\\
 & = & -\frac{\sigma_{2}^{ii}[\sum\limits _{k}\lambda_{k}\sigma_{2}^{kk}u_{kki}+\sigma_{1}f_{i}]^{2}}{(\sigma_{1}f-\sigma_{3})^{2}}.
\end{eqnarray}
\begin{equation}
IV:=\frac{2\sigma_{2}^{ii}(\sigma_{1})_{i}f_{i}}{\sigma_{1}f-\sigma_{3}}+\frac{\sigma_{1}\sigma_{2}^{ij}f_{ij}}{\sigma_{1}f-\sigma_{3}}+\frac{f\Delta f-\sigma_{3}^{ij}f_{ij}}{\sigma_{1}f-\sigma_{3}}.
\end{equation}
Then we use following equation to replace the $u_{kkk}$ term above,
\begin{equation}
\sum_{i\neq k}\sigma_{2}^{ii}u_{iik}+\sigma_{2}^{kk}u_{kkk}=f_{k}.
\end{equation}
We have 
\begin{eqnarray}
I & = & \frac{-2\sum\limits _{k\neq p,i\neq p,i\neq k}\sigma_{2}^{ii}\lambda_{i}u_{kki}u_{ppi}}{\sigma_{1}f-\sigma_{3}}\\
 &  & -\frac{2\sum\limits _{p}\sum\limits _{k\neq p}\sigma_{2}^{kk}u_{kkp}\sigma_{2}^{pp}u_{ppp}}{\sigma_{1}f-\sigma_{3}}\\
 & = & \frac{-2\sum\limits _{k\neq p,i\neq p,i\neq k}\sigma_{2}^{ii}\lambda_{i}u_{kki}u_{ppi}}{\sigma_{1}f-\sigma_{3}}\\
 &  & -\frac{2\sum\limits _{p}\sum\limits _{k\neq p}\sigma_{2}^{kk}u_{kkp}(f_{p}-\sum\limits _{i\neq p}\sigma_{2}^{ii}u_{iip})}{\sigma_{1}f-\sigma_{3}}\\
 & = & \frac{-2\sum_{k\neq p,i\neq p,i\neq k}\sigma_{2}^{ii}\lambda_{i}u_{kki}u_{ppi}}{\sigma_{1}f-\sigma_{3}}\\
 &  & -\frac{2\sum_{k\neq p}\sigma_{2}^{kk}u_{kkp}f_{p}}{\sigma_{1}f-\sigma_{3}}+2\sum_{p}\frac{\sum_{k\neq p}(\sigma_{2}^{kk})^{2}(u_{kkp})^{2}}{\sigma_{1}f-\sigma_{3}}\\
 &  & +2\frac{\sum\limits _{k\neq p,i\neq p,i\neq k}\sigma_{2}^{kk}\sigma_{2}^{ii}u_{iip}u_{kkp}}{\sigma_{1}f-\sigma_{3}}\\
 & = & \frac{2\sum\limits _{k\neq p,i\neq p,i\neq k}(-\sigma_{2}^{ii}\lambda_{i}+\sigma_{2}^{kk}\sigma_{2}^{pp})u_{kki}u_{ppi}}{\sigma_{1}f-\sigma_{3}}\\
 &  & -\frac{2\sum\limits _{k\neq p}\sigma_{2}^{kk}u_{kkp}f_{p}}{\sigma_{1}f-\sigma_{3}}+2\sum_{p}\frac{\sum\limits _{k\neq p}(\sigma_{2}^{kk})^{2}(u_{kkp})^{2}}{\sigma_{1}f-\sigma_{3}}.
\end{eqnarray}
And 
\begin{eqnarray}
II & = & \frac{2\sum\limits _{k\neq p,i\neq p,i\neq k}\sigma_{2}^{ii}\lambda_{i}u_{kpi}^{2}}{\sigma_{1}f-\sigma_{3}}\\
 &  & +\frac{2\sum_{k\neq p}\sigma_{2}^{pp}\sigma_{2}^{kk}u_{kpp}^{2}}{\sigma_{1}f-\sigma_{3}}\\
 & = & \frac{4fu_{123}^{2}}{\sigma_{1}f-\sigma_{3}}+\frac{2\sum_{k\neq p}\sigma_{2}^{pp}\sigma_{2}^{kk}u_{kpp}^{2}}{\sigma_{1}f-\sigma_{3}}.
\end{eqnarray}
\begin{eqnarray}
III & = & -\frac{\sigma_{2}^{ii}[\sum_{k\neq i}\lambda_{k}\sigma_{2}^{kk}u_{kki}+\lambda_{i}(f_{i}-\sigma_{2}^{kk}u_{kki})+\sigma_{1}f_{i}]^{2}}{(\sigma_{1}f-\sigma_{3})^{2}}\\
 & = & -\frac{\sigma_{2}^{ii}[\sum_{k\neq i}(\lambda_{k}-\lambda_{i})\sigma_{2}^{kk}u_{kki}+\lambda_{i}f_{i}+\sigma_{1}f_{i}]^{2}}{(\sigma_{1}f-\sigma_{3})^{2}}.
\end{eqnarray}
There is an identity 
\begin{equation}
\sigma_{1}f-\sigma_{3}=\sigma_{2}^{11}\sigma_{2}^{22}\sigma_{2}^{33}.
\end{equation}
Then we get 
\begin{eqnarray}
III & = & -\frac{(\lambda_{3}-\lambda_{1})^{2}\sigma_{2}^{33}u_{331}^{2}}{(\sigma_{1}f-\sigma_{3})\sigma_{2}^{22}}-2\frac{(\lambda_{3}-\lambda_{1})(\lambda_{2}-\lambda_{1})u_{221}u_{331}}{\sigma_{1}f-\sigma_{3}}\\
 &  & -\frac{(\lambda_{2}-\lambda_{1})^{2}\sigma_{2}^{22}u_{221}^{2}}{(\sigma_{1}f-\sigma_{3})\sigma_{2}^{33}}+\frac{2(\lambda_{1}+\sigma_{1})f_{1}(\lambda_{1}-\lambda_{3})u_{331}}{(\sigma_{1}f-\sigma_{3})\sigma_{2}^{22}}\\
 &  & +\frac{2(\lambda_{1}+\sigma_{1})f_{1}(\lambda_{1}-\lambda_{2})u_{221}}{(\sigma_{1}f-\sigma_{3})\sigma_{2}^{33}}-\frac{\sigma_{2}^{11}(\lambda_{1}+\sigma_{1})^{2}f_{1}^{2}}{(\sigma_{1}f-\sigma_{3})^{2}}\\
 &  & -\frac{(\lambda_{3}-\lambda_{2})^{2}\sigma_{2}^{33}u_{332}^{2}}{(\sigma_{1}f-\sigma_{3})\sigma_{2}^{11}}-\frac{2(\lambda_{3}-\lambda_{2})(\lambda_{1}-\lambda_{2})u_{112}u_{332}}{\sigma_{1}f-\sigma_{3}}\\
 &  & -\frac{(\lambda_{1}-\lambda_{2})\sigma_{2}^{11}u_{112}^{2}}{(\sigma_{1}f-\sigma_{3})\sigma_{2}^{33}}+\frac{2(\lambda_{2}+\sigma_{1})f_{2}(\lambda_{2}-\lambda_{3})u_{332}}{(\sigma_{1}f-\sigma_{3})\sigma_{2}^{11}}\\
 &  & +\frac{2(\lambda_{2}-\lambda_{1})u_{112}f_{2}(\sigma_{1}+\lambda_{2})}{(\sigma_{1}f-\sigma_{3})\sigma_{2}^{33}}-\frac{\sigma_{2}^{22}(\lambda_{2}+\sigma_{1})^{2}f_{2}^{2}}{(\sigma_{1}f-\sigma_{3})^{2}}\\
 &  & -\frac{(\lambda_{1}-\lambda_{3})^{2}\sigma_{2}^{11}u_{113}^{2}}{(\sigma_{1}f-\sigma_{3})\sigma_{2}^{22}}-\frac{(\lambda_{2}-\lambda_{3})^{2}\sigma_{2}^{22}u_{223}^{2}}{(\sigma_{1}f-\sigma_{3})\sigma_{2}^{11}}\\
 &  & -\frac{2(\lambda_{1}-\lambda_{3})(\lambda_{2}-\lambda_{3})u_{113}u_{223}}{\sigma_{1}f-\sigma_{3}}-\frac{\sigma_{2}^{33}(\lambda_{3}+\sigma_{1})^{2}f_{3}^{2}}{(\sigma_{1}f-\sigma_{3})^{2}}\\
 &  & +\frac{2(\lambda_{3}-\lambda_{1})u_{113}(\lambda_{3}+\sigma_{1})f_{3}}{(\sigma_{1}f-\sigma_{3})\sigma_{2}^{22}}+\frac{2(\lambda_{3}-\lambda_{2})u_{223}(\lambda_{3}+\sigma_{1})f_{3}}{(\sigma_{1}f-\sigma_{3})\sigma_{2}^{11}}.
\end{eqnarray}
We can also compute $I$ and $II$ explicitly, 
\begin{eqnarray}
I & = & -\frac{2\sum\limits _{k\neq1}\sigma_{2}^{kk}u_{kk1}f_{1}}{\sigma_{1}f-\sigma_{3}}-\frac{2\sum\limits _{k\neq2}\sigma_{2}^{kk}u_{kk2}f_{2}}{\sigma_{1}f-\sigma_{3}}-\frac{2\sum\limits _{k\neq3}\sigma_{2}^{kk}u_{kk3}f_{3}}{\sigma_{1}f-\sigma_{3}}\\
 &  & +2\frac{(\sigma_{2}^{22})^{2}(u_{221})^{2}}{\sigma_{1}f-\sigma_{3}}+2\frac{(\sigma_{2}^{33})^{2}(u_{331})^{2}}{\sigma_{1}f-\sigma_{3}}+2\frac{(\sigma_{2}^{11})^{2}(u_{112})^{2}}{\sigma_{1}f-\sigma_{3}}\\
 &  & +2\frac{(\sigma_{2}^{33})^{2}(u_{332})^{2}}{\sigma_{1}f-\sigma_{3}}+2\frac{(\sigma_{2}^{11})^{2}(u_{113})^{2}}{\sigma_{1}f-\sigma_{3}}+2\frac{(\sigma_{2}^{22})^{2}(u_{223})^{2}}{\sigma_{1}f-\sigma_{3}}\\
 &  & +\frac{4(-\sigma_{2}^{11}\lambda_{1}+\sigma_{2}^{22}\sigma_{2}^{33})u_{221}u_{331}}{\sigma_{1}f-\sigma_{3}}+\frac{4(-\sigma_{2}^{22}\lambda_{2}+\sigma_{2}^{11}\sigma_{2}^{33})u_{112}u_{332}}{\sigma_{1}f-\sigma_{3}}\\
 &  & +\frac{4(-\sigma_{2}^{33}\lambda_{3}+\sigma_{2}^{22}\sigma_{2}^{11})u_{223}u_{113}}{\sigma_{1}f-\sigma_{3}}.
\end{eqnarray}
And
\begin{eqnarray}
II & = & \frac{4fu_{123}^{2}}{\sigma_{1}f-\sigma_{3}}+\frac{2\sum\limits _{k\neq p}\sigma_{2}^{pp}\sigma_{2}^{kk}u_{kpp}^{2}}{\sigma_{1}f-\sigma_{3}}\\
 & = & \frac{4fu_{123}^{2}}{\sigma_{1}f-\sigma_{3}}+\frac{2\sigma_{2}^{11}\sigma_{2}^{22}u_{211}^{2}}{\sigma_{1}f-\sigma_{3}}+\frac{2\sigma_{2}^{11}\sigma_{2}^{33}u_{311}^{2}}{\sigma_{1}f-\sigma_{3}}\\
 &  & +\frac{2\sigma_{2}^{22}\sigma_{2}^{11}u_{122}^{2}}{\sigma_{1}f-\sigma_{3}}+\frac{2\sigma_{2}^{22}\sigma_{2}^{33}u_{322}^{2}}{\sigma_{1}f-\sigma_{3}}\\
 &  & +\frac{2\sigma_{2}^{33}\sigma_{2}^{11}u_{133}^{2}}{\sigma_{1}f-\sigma_{3}}+\frac{2\sigma_{2}^{33}\sigma_{2}^{22}u_{233}^{2}}{\sigma_{1}f-\sigma_{3}}.
\end{eqnarray}
We also have from Lemma \ref{lem1},
\begin{eqnarray}
IV & \geq & \frac{2\sigma_{2}^{11}(\sigma_{1})_{1}f_{1}}{\sigma_{1}f-\sigma_{3}}+\frac{2\sigma_{2}^{22}(\sigma_{1})_{2}f_{2}}{\sigma_{1}f-\sigma_{3}}+\frac{2\sigma_{2}^{33}(\sigma_{1})_{3}f_{3}}{\sigma_{1}f-\sigma_{3}}\\
 &  & -C\sigma_{1}+\frac{f_{u_{i}}[\sigma_{1}f_{i}+f(\sigma_{1})_{i}-(\sigma_{3})_{i}]}{\sigma_{1}f-\sigma_{3}}\\
 & = & \frac{2\sigma_{2}^{11}(\frac{f_{1}-\sigma_{2}^{22}u_{221}-\sigma_{2}^{33}u_{331}}{\sigma_{2}^{11}}+u_{221}+u_{331})f_{1}}{\sigma_{1}f-\sigma_{3}}\\
 &  & +\frac{2\sigma_{2}^{22}(\frac{f_{2}-\sigma_{2}^{11}u_{112}-\sigma_{2}^{33}u_{332}}{\sigma_{2}^{22}}+u_{112}+u_{332})f_{2}}{\sigma_{1}f-\sigma_{3}}\\
 &  & +\frac{2\sigma_{2}^{33}(\frac{f_{3}-\sigma_{2}^{22}u_{223}-\sigma_{2}^{11}u_{113}}{\sigma_{2}^{33}}+u_{223}+u_{113})f_{3}}{\sigma_{1}f-\sigma_{3}}\\
 &  & -C\sigma_{1}+f_{u_{i}}V_{i}\\
 & = & \frac{2((\lambda_{2}-\lambda_{1})u_{221}+(\lambda_{3}-\lambda_{1})u_{331})f_{1}}{\sigma_{1}f-\sigma_{3}}\\
 &  & +\frac{2((\lambda_{1}-\lambda_{2})u_{112}+(\lambda_{3}-\lambda_{2})u_{332})f_{2}}{\sigma_{1}f-\sigma_{3}}\\
 &  & +\frac{2((\lambda_{2}-\lambda_{3})u_{223}+(\lambda_{1}-\lambda_{3})u_{113})f_{3}}{\sigma_{1}f-\sigma_{3}}\\
 &  & -C\sigma_{1}+f_{u_{i}}V_{i}.
\end{eqnarray}
So we combine above inequalities 
\begin{eqnarray}
\sigma_{2}^{ij}V{}_{ij} & \geq & \text{-}\frac{4\lambda_{2}u_{221}f_{1}\sigma_{2}^{22}}{(\sigma_{1}f-\sigma_{3})\sigma_{2}^{33}}-\frac{4\lambda_{3}u_{331}f_{1}\sigma_{2}^{33}}{(\sigma_{1}f-\sigma_{3})\sigma_{2}^{22}}\\
 &  & \text{-}\frac{4\lambda_{1}u_{112}f_{2}\sigma_{2}^{11}}{(\sigma_{1}f-\sigma_{3})\sigma_{2}^{33}}-\frac{4\lambda_{3}u_{332}f_{2}\sigma_{2}^{33}}{(\sigma_{1}f-\sigma_{3})\sigma_{2}^{11}}\\
 &  & -\frac{4\lambda_{1}u_{113}f_{3}\sigma_{2}^{11}}{(\sigma_{1}f-\sigma_{3})\sigma_{2}^{22}}-\frac{4\lambda_{2}u_{223}f_{3}\sigma_{2}^{22}}{(\sigma_{1}f-\sigma_{3})\sigma_{2}^{11}}\\
 &  & +\frac{2\sigma_{2}^{22}\sigma_{2}^{33}u_{221}u_{331}}{\sigma_{1}f-\sigma_{3}}+\frac{2\sigma_{2}^{11}\sigma_{2}^{33}u_{112}u_{332}}{\sigma_{1}f-\sigma_{3}}\\
 &  & +\frac{2\sigma_{2}^{11}\sigma_{2}^{22}u_{113}u_{223}}{\sigma_{1}f-\sigma_{3}}+\frac{4fu_{123}^{2}}{\sigma_{1}f-\sigma_{3}}\\
 &  & +\frac{\sigma_{2}^{33}[4f+(\sigma_{2}^{22})^{2}]u_{331}^{2}}{(\sigma_{1}f-\sigma_{3})\sigma_{2}^{22}}+\frac{\sigma_{2}^{22}[4f+(\sigma_{2}^{33})^{2}]u_{221}^{2}}{(\sigma_{1}f-\sigma_{3})\sigma_{2}^{33}}\\
 &  & +\frac{\sigma_{2}^{33}[4f+(\sigma_{2}^{11})^{2}]u_{332}^{2}}{(\sigma_{1}f-\sigma_{3})\sigma_{2}^{11}}+\frac{\sigma_{2}^{11}[4f+(\sigma_{2}^{33})^{2}]u_{112}^{2}}{(\sigma_{1}f-\sigma_{3})\sigma_{2}^{33}}\\
 &  & +\frac{\sigma_{2}^{11}[4f+(\sigma_{2}^{22})^{2}]u_{113}^{2}}{(\sigma_{1}f-\sigma_{3})\sigma_{2}^{22}}+\frac{\sigma_{2}^{22}[4f+(\sigma_{2}^{11})^{2}]u_{223}^{2}}{(\sigma_{1}f-\sigma_{3})\sigma_{2}^{11}}\\
 &  & -C\sigma_{1}+f_{u_{i}}V_{i}.
\end{eqnarray}
Square these terms,
\begin{eqnarray}
\sigma_{2}^{ij}V{}_{ij} & \geq & \frac{\sigma_{2}^{22}\sigma_{2}^{33}(u_{331}+u_{221})^{2}}{\sigma_{1}f-\sigma_{3}}+\frac{\sigma_{2}^{33}\sigma_{2}^{11}(u_{332}+u_{112})^{2}}{\sigma_{1}f-\sigma_{3}}\\
 &  & +\frac{\sigma_{2}^{11}\sigma_{2}^{22}(u_{113}+u_{223})^{2}}{\sigma_{1}f-\sigma_{3}}+\frac{4fu_{123}^{2}}{\sigma_{1}f-\sigma_{3}}\\
 &  & +\frac{4f\sigma_{2}^{33}(u_{331}-\frac{f_{1}}{2f}\lambda_{3})^{2}}{(\sigma_{1}f-\sigma_{3})\sigma_{2}^{22}}+\frac{4f\sigma_{2}^{22}(u_{221}-\frac{f_{1}}{2f}\lambda_{2})^{2}}{(\sigma_{1}f-\sigma_{3})\sigma_{2}^{33}}\\
 &  & +\frac{4f\sigma_{2}^{33}(u_{332}-\frac{f_{2}}{2f}\lambda_{3})^{2}}{(\sigma_{1}f-\sigma_{3})\sigma_{2}^{11}}+\frac{4f\sigma_{2}^{11}(u_{112}-\frac{f_{2}}{2f}\lambda_{1})^{2}}{(\sigma_{1}f-\sigma_{3})\sigma_{2}^{33}}\\
 &  & +\frac{4f\sigma_{2}^{11}(u_{113}-\frac{f_{3}}{2f}\lambda_{1})^{2}}{(\sigma_{1}f-\sigma_{3})\sigma_{2}^{22}}+\frac{4f\sigma_{2}^{22}(u_{223}-\frac{f_{3}}{2f}\lambda_{2})^{2}}{(\sigma_{1}f-\sigma_{3})\sigma_{2}^{11}}\\
 &  & -C\sigma_{1}+f_{u_{i}}V{}_{i}.
\end{eqnarray}
In above inequality we have used 
\begin{eqnarray}
\frac{\sigma_{2}^{22}f_{3}^{2}\lambda_{2}^{2}}{(\sigma_{1}f-\sigma_{3})\sigma_{2}^{11}} & \leq & \frac{C\sigma_{2}^{22}\lambda_{3}^{2}\lambda_{2}^{2}}{(\sigma_{1}f-\sigma_{3})\sigma_{2}^{11}}\\
 & = & \frac{C\sigma_{2}^{22}(f-\lambda_{1}\sigma_{2}^{11})^{2}}{(\sigma_{1}f-\sigma_{3})\sigma_{2}^{11}}\\
 & \leq & C\sigma_{1}+C\frac{\sigma_{2}^{22}\sigma_{2}^{11}\lambda_{1}^{2}}{\sigma_{1}f-\sigma_{3}}\\
 & \leq & C\sigma_{1}+C\frac{\sigma_{2}^{22}(f\lambda_{1}-\sigma_{3})}{\sigma_{1}f-\sigma_{3}}\\
 & \leq & C\sigma_{1}.
\end{eqnarray}

So we have proved this lemma.
\end{proof}

\subsection{Proof of (\ref{eq:V-1})}

Denote $\rho(x)=5^{2}-|x|^{2}$, and $M_{r}:=\sup_{B_{1}}\sigma_{1}$.
We consider test function 
\begin{equation}
\varphi(x)=4\log\rho(x)+h(\frac{|Du|^{2}}{2})+\log\log\{\frac{(\sigma_{1}f-\sigma_{3})\log^{\frac{1}{2}}M_{5}}{M_{5}\log(\sigma_{1}f-\sigma_{3})},20\}
\end{equation}
in $B_{5}$, where $h(t)=-\frac{1}{2}\log(1-\frac{t}{2\max|Du|^{2}+1})$. 

Assume $\varphi$ attains its maximum point $z_{0}\in\Omega$ , and
$\frac{(\sigma_{1}f-\sigma_{3})\log^{\frac{1}{2}}M_{5}}{M_{5}\log(\sigma_{1}f-\sigma_{3})}>20$.
Then $\varphi$ must attained its maximum point in the ring $B_{5}$.
Moreover, we may always assume $\frac{(\sigma_{1}f-\sigma_{3})\log^{\frac{1}{2}}M_{5}}{M_{5}\log(\sigma_{1}f-\sigma_{3})}$
sufficiant large. We choose coordinate frame $\{e_{1},e_{2},e_{3}\}$,
such that $D^{2}u$ is diagonalized at this point. We denote $M=\log\frac{M_{5}}{\log^{\frac{1}{2}}M_{5}}$.

By maximum principle, at $z_{0}$, we have

\begin{equation}
0=\varphi_{i}=\frac{4\rho_{i}}{\rho}+h^{\prime}u_{k}u_{ki}+\frac{V_{i}-\frac{V_{i}}{V}}{V-M-\log V}.\label{eq:fp-1}
\end{equation}
and

\begin{eqnarray}
\varphi_{ij} & = & \frac{4\rho_{ij}}{\rho}-\frac{4\rho_{i}\rho_{j}}{\rho^{2}}+h^{\prime\prime}u_{k}u_{ki}u_{l}u_{lj}+h^{\prime}(u_{kj}u_{ki}+u_{k}u_{kij})\\
 &  & +\frac{V_{ij}-\frac{V_{ij}}{V}+\frac{V_{i}V_{j}}{V^{2}}}{V-M-\log V}-\frac{V_{i}V_{j}}{(V-M-\log V)^{2}}(1-\frac{1}{V})^{2}.
\end{eqnarray}

Contracting with $\sigma_{2}^{ij}:=\frac{\partial\sigma_{2}(D^{2}u)}{\partial u_{ij}}$,
we get

\begin{eqnarray*}
\sigma_{2}^{ij}\varphi_{ij} & \geq & -8\frac{\sum\sigma_{2}^{ii}}{\rho}-16\frac{\sigma_{2}^{ii}x_{i}^{2}}{\rho^{2}}+h^{\prime\prime}\sigma_{2}^{ii}u_{i}^{2}u_{ii}^{2}+h^{\prime}(\sigma_{1}f-3\sigma_{3})\\
 &  & +h^{\prime}u_{i}f_{i}+\frac{\sigma_{2}^{ii}V_{ii}(1-\frac{1}{V})}{V-M-\log V}-\frac{\sigma_{2}^{ii}V_{i}^{2}}{(V-M-\log V)^{2}}(1-\frac{1}{V})^{2}.
\end{eqnarray*}

By (\ref{eq:subV}), 
\begin{equation}
\sigma_{2}^{ii}V_{ii}\geq-C\sigma_{1}+f_{u_{i}}V_{i}.
\end{equation}

And by (\ref{eq:fp-1}), 
\begin{equation}
f_{u_{i}}\frac{V_{i}-\frac{V_{i}}{V}}{V-M-\log V}+h^{\prime}u_{i}f_{i}\geq-C(\frac{1}{\rho}+h^{\prime})
\end{equation}

Then combine this inequality and (\ref{eq:fp-1}), we have  
\begin{eqnarray}
\sigma_{2}^{ij}\varphi_{ij} & \geq & -8\frac{\sum\sigma_{2}^{ii}}{\rho}-48\frac{\sigma_{2}^{ii}x_{i}^{2}}{\rho^{2}}+(h^{\prime\prime}-2(h^{\prime})^{2})\sigma_{2}^{ii}u_{i}^{2}u_{ii}^{2}\label{eq:var}\\
 &  & +\frac{h^{\prime}}{2}(\sigma_{1}f-3\sigma_{3}).
\end{eqnarray}

If we choose $h(t)=-\frac{1}{2}\log(1-\frac{t}{2\max|Du|^{2}})$ such
that $h^{\prime\prime}-2(h^{\prime})^{2}\geq0$, we have 
\begin{eqnarray}
\sigma_{2}^{ij}\varphi_{ij} & \geq & -4\frac{\sum\sigma_{2}^{ii}}{\rho}-1200\frac{\sum\sigma_{2}^{ii}}{\rho^{2}}+h^{\prime}\frac{M_{5}\log(\sigma_{1}f-\sigma_{3})}{\log^{\frac{1}{2}}M_{5}}\\
 & \geq & -1300\frac{\sum\sigma_{2}^{ii}}{\rho^{2}}+h^{\prime}\frac{M_{5}\log^{\frac{1}{2}}M_{5}}{2}.
\end{eqnarray}
So we get
\begin{equation}
\rho^{4}(z_{0})\log\frac{(\sigma_{1}f-\sigma_{3})(z_{0})\log^{\frac{1}{2}}M_{5}}{M_{5}\log(\sigma_{1}f-\sigma_{3})(z_{0})}\le3\rho^{4}(z_{0})\log M_{5}\leq C.
\end{equation}

Finally, we obtain in $B_{4}$ 
\begin{equation}
(\sigma_{1}f-\sigma_{3})(x)\leq\frac{M_{5}\log(\sigma_{1}f-\sigma_{3})}{\log^{\frac{1}{2}}M_{5}}.
\end{equation}

\end{document}